\documentclass[12pt]{article}

\usepackage{fullpage}
\usepackage{amssymb}
\usepackage{amsmath}
\usepackage{amsthm}
\usepackage{xypic}
\usepackage{setspace}
\usepackage{array}
\usepackage{rotating}
\usepackage{color}
\usepackage[dvipsnames,usenames]{xcolor}
\usepackage{pdflscape}
\usepackage{multirow}
\usepackage{hyperref}

\newtheorem{theorem}{Theorem}[section]

\newtheorem{problem}[theorem]{Open Problem}

\newcommand{\cref}[1]{Corollary~\textup{\ref{cor:#1}}}

\newcommand{\comment}[1]{}

\DeclareMathOperator{\SR}{SR}

\begin{document}

\title{Update on Sidon-Ramsey Numbers}

\author{Manuel A. Espinosa-Garc\'ia\\
Centro de Ciencias Matem\'aticas\\
Universidad Nacional Aut\'onoma de M\'exico\\
Morelia, Mich. M\'exico \\
and \\
Daniel Pellicer\\
Centro de Ciencias Matem\'aticas\\
Universidad Nacional Aut\'onoma de M\'exico\\
Morelia, Mich. M\'exico
}

\date{ \today }
\maketitle

\begin{abstract}
We provide two new exact Sidon-Ramsey numbers to the list known so far. We also improve the upper bounds of the next two Sidon-Ramsey numbers. In doing so, we comment on the tendencies we found on the Sidon-Ramsey partitions that were studied to obtain these results.



\end{abstract}

\section{Background}


A \textit{Sidon set} is a subset $S$ of an additive group such that all pairwise sums are distinct, i.e., if $a+b=c+d$, for some $a,b,c,d\in S$, then $\{a,b\}=\{c,d\}$. The most studied problems in this area are the following. One of them is to find the maximum size of a Sidon set contained in $[n]:=\{1,2,\dots,n\}$ as subset of $\mathbb{Z}$, and the maximal density of a Sidon set in the set of positive integers in the group $\mathbb{Z}$. $F_{2}(n)$ is defined as the size of the largest Sidon set contained in the interval $[n]$, and it is called the \textit{Sidon number} of $n$. It is known that 
\[n^{1/2}(1-o(1))\leq F_{2}(n)\leq n^{1/2}+O(n^{1/4}).\]
The lower bound is inferred from a construction made by Singer of Sidon sets (see \cite{Sin38}), and the upper bound was proved by Erd\H os and Turán (see \cite{ET41}). Some of the first exact values known for $F_{2}(n)$ are included in Table \ref{table:sidonvalues}.

\begin{table}[h]\centering
\begin{tabular}{cccccccc}
	$n$&$[1]$&$[2,3]$&$[4,6]$&$[7,11]$&$[12,17]$&$[18,25]$&$[26,34]$\\

	$F_{2}(n)$&1&2&3&4&5&6&7\\
	
	&&&&&&&\\
	
	$n$&$[35,44]$&$[45,55]$&$[56,72]$&$[73,85]$&$[86,106]$&$[107,127]$&$[128,151]$\\
	
	$F_{2}(n)$&8&9&10&11&12&13&14
\end{tabular}
\caption{Values of $F_{2}(n)$ for small $n$. \label{table:sidonvalues}}

\end{table}

The density problem for the Sidon numbers consists of maximizing the length of an interval that can be partitioned into $k$ Sidon sets, we call such division a \textit{Sidon-Ramsey partition}. $\SR(k)$ is defined as the minimum $n$ such that there is no Sidon-Ramsey partition of $[n]$ in $k$ parts, and it is called the \textit{Sidon-Ramsey number} of $k$. These numbers were introduced by Liang, Li, Xiu y Xu in \cite{LLXX13}. The Sidon-Ramsey numbers satisfy
\[k^{2}-O(k^{c})\le SR(k)\le k^{2}+Ck^{3/2}+O(k),\]
where $c\le 1.525$ and $C\le 1.996$ (see \cite{EMRS23}). Table \ref{table:sidonramseyvalues} lists the previously known values of $\SR(k)$, that is, those for $k\leq 5$. It also includes the bounds for the next values of $k$ given in \cite{XLL18}.

\begin{table}[h]\centering
	
	\begin{tabular}{ccccc}
		$\SR(1)=3$, & $\SR(2)=8$, & $\SR(3)=15$, & $\SR(4)=25$, & $\SR(5)=36$,\\
	\end{tabular}

	\vspace{0.15in}

	\begin{tabular}{c|c|c|c|c|c|c}
		$k$ & 6 & 7 & 8 & 9 & 10 & 11\\
		\hline
		upper bound for $SR(k)$ &55&70&97&118&141&166\\
		lower bound for $SR(k)$ &50&65&81&97&114&133\\
	\end{tabular}
	
	\caption{$\SR(k)$ for $k\leq 5$, and $\SR(k)$ bounds for $6\le k\le 11$. \label{table:sidonramseyvalues}}
\end{table}

In this paper we establish the values of the next two Sidon-Ramsey numbers; they are given in Section \ref{s:exact}. Before that, in Section \ref{s:routines} we describe the procedures implemented in python to find Sidon sets on a given finite subset of $\mathbb{Z}$. Then, in Section \ref{s:Further} we improve the bounds of $SR(8)$ and of $SR(9)$. We conclude with some  remarks and open problems in Section \ref{s:Conclusions}.

\section{Description of routines}\label{s:routines}


In this section we describe the procedures we implemented in python in order to find Sidon sets of prestablished sizes on given subsets of $\mathbb{Z}$. When determining the exact values of Sidon-Ramsey numbers we used both procedures in separate computers to validate the results.

For convenience, in this paper we will abbreviate `Sidon set with $k$ elements' by $k$-SS.

\subsection{Using Sidon sets on smaller subsets of $\mathbb{Z}$}

This algorithm finds all Sidon sets with at most $k$ elements in the set $[n]$. We build the Sidon sets recursively, using Sidon sets in $[m]$ to build Sidon sets in $[m+1]$.

First, for $k \ge 1$ we consider all Sidon sets in the interval $[1]$ with at most $k$ element, and denote this set by $S_{1,k}$. Notice that $S_{1,k}$ consists exclusively of $\varnothing$ and $\{1\}$. Recursively, we build the sets with all the Sidon sets with at most $k$ elements in $[t]$, denoted by $S_{t,k}$, as follows:

\begin{enumerate}
	\item We add all the elements in $S_{t-1,k}$ to $S_{t,k}$.
	
	\item To each Sidon set in $S_{t-1,k}$ with at most $k-1$ elements we add the element $t$. If the new set is a Sidon set, we add it to $S_{t,k}$.
\end{enumerate}


\subsection{Directly constructing all Sidon sets of a given size on a given subset of $\mathbb{Z}$}

Next we explain how we find all $k$-SS's in the set $[n]$. It suffices to determine a way to obtain all $k$-SS's that contain $1$ and $n$, since all others will be translates of some $k$-SS's obtained by the same procedure in subsets $[m]$ for some $m<n$.

Let $X=\{x_1,\dots,x_k\}$ with $x_1=1$ and $x_k=n$ be a $k$-SS where $x_i < x_j$ if $i<j$. Then $X$ is completely determined by the $(k-1)$-tuple $(y_1,\dots,y_k) := (x_2-x_1,x_3-x_2,\dots,x_k-x_{k-1})$. The numbers $y_k$ satisfy the following properties.
\begin{enumerate}
 \item If $i\ne j$ then $y_i \ne y_j$. More generally, if $i_1 \le i_2$, $i_3 \le i_4$ and $\{i_1,i_2\} \ne \{i_3,i_4\}$ then
  $$\sum_{j=i_1}^{i_2} y_j \ne \sum_{j=i_3}^{i_4} y_j.$$
 \item $\displaystyle{\sum_{i=1}^{k-1}} y_i = n-1$.
\end{enumerate}
Conversely, any set $\{y_1,\dots,y_{k-1}\}$ satisfying the above properties induces the Sidon set $\{1,1+y_1,1+y_1+y_2,\dots,1+\sum_{i=1}^{k-1} y_i\}$.

Our strategy consists of two steps.
\begin{itemize}
 \item Find all ordered sets $(y_1,\dots,y_{k-1})$ of numbers such that their sum is $n-1$ and that $y_i<y_{i+1}$ for all $i$.
 \item For each of the sets in the previous step determine which of the permutations of its entries induces a Sidon set (and so we only need to verify that they satisfy the first item above).
\end{itemize}

Clearly this procedure becomes too slow when the values of $k$ and $n$ increase. As an example, our implementation in python took less than a second to find the 96 $8$-SS in [38] that contain $1$ and $38$, and it took a couple of seconds to find the 195 $9$-SS in [50] that contain $1$ and $50$. When asked to determine all $13$-SS in the set $[107]$, followed by those in $[108]$ and by those in $[109]$ the three results were ready only after $19$ hours (the outcome is that there are two that contain $1$ and $107$ and none that contain either $1$ and $108$, or $1$ and $109$). As expected, the computation of the $13$-SS in the sets $[110]$ and $[111]$ that contain both endpoints of the intervals was even lengthier.

\section{New exact numbers}\label{s:exact}


In this section we establish the value of two Sidon sets that were not previously known.

\begin{theorem}
The number $SR(6)$ is $51$.
\end{theorem}

\begin{proof}
The lower bound is given by the following partition of $[50]$:
\begin{align*}
&\{4,6,14,28,29,33,40,46,49\},&
&\{2,5,11,18,22,23,37,45,47\},\\
&\{3,7,8,24,30,32,39,42\},&
&\{1,13,15,26,31,34,35,41\},\\
&\{9,12,19,21,27,43,44,48\},&
&\{10,16,17,20,25,36,38,50\}.
\end{align*}

We verified the upper bound in two distinct ways.

First, we found all triples of mutually disjoint $9$-SS's in the set $[51]$. There are $12,094$ triples, and for each of them we did an exhaustive search to determine that it cannot be completed to a Sidon-Ramsey partition of $[51]$. The search was carried out by finding all $8$-SS's in the complement in $[51]$ of each triple and determining whether four of them are mutually disjoint. We also verified that there is no quadruple of mutually disjoint $9$-SS's, and it was previously known that there are no $10$-SS's in $[51]$ (see Table \ref{table:sidonvalues}). Therefore there are no partitions of $[51]$ into $6$ Sidon sets that do not have three $9$-SS's or three $8$-SS's.

Second, with the same techniques described above (with pairs of $9$-SS's instead of triples) we made an exhaustive search to establish that there is only one Sidon-Ramsey partition of $[50]$ (the one in the displayed equation). If there existed a Sidon-Ramsey partition of $[51]$ then it would be possible to obtain that one from the one in $[50]$ by adding $51$ to one of the parts. However, none of the parts remains a Sidon set when the number $51$ is added.
\end{proof}

\begin{theorem}
The number $SR(7)$ is $66$.
\end{theorem}

\begin{proof}
The following partition is a witness of the lower bound:
\begin{align*}
&\{1, 3, 15, 22, 30, 33, 46, 50, 55, 56\},& \{2, 9, 17, 21, 26, 27, 47, 49, 60, 63\},\\
&\{4, 7, 13, 23, 24, 28, 36, 54, 61\},& \{5, 12, 14, 20, 34, 44, 45, 57, 62\},\\
&\{6, 10, 16, 29, 38, 41, 43, 58, 59\},& \{8, 25, 31, 32, 35, 40, 51, 53, 65\},\\
&\{11, 18, 19, 37, 39, 42, 48, 52, 64\}.&
\end{align*}

The upper bound was verified by finding all triples of mutually disjoint $10$-SS's in the set $[66]$. There are $1601160$ such triples, and for each of them we did an exhaustive search to determine that the triple cannot be completed to a Sidon-Ramsey partition of $[66]$ with the addition of four disjoint $9$-SS's. Besides, there are $12435$ and $0$ quadruples and quintuples of mutually disjoint $10$-SS's in the set $[66]$, and none of them can be extended to a Sidon-Ramsey partition of $[66]$.
\end{proof}

\section{Further results}\label{s:Further}


Our techniques and computational resources seem not to be enough to determine $SR(8)$, but we improved the bounds in Table \ref{table:sidonramseyvalues} as shown in the following results.

\begin{theorem}
	The Sidon-Ramsey number with $k=8$ satisfies $\SR(8)\leq 86$.
\end{theorem}

\begin{proof}
From Table \ref{table:sidonvalues} we know that the size of each part of a Sidon set in $[86]$ is at most $12$. There are two $12$-SS's in $[86]$ and they have non empty intersection, forcing any partition of $[86]$ into eight Sidon sets to have at most one $12$-SS. Also, we found the $102484$ $11$-SS's in $[86]$. In order to find partitions of $[86]$ into eight Sidon sets with one of them of size $12$, we needed to complete the $12$-SS with at least four mutually disjoint $11$-SS's. There are $3266$ quadruples of disjoint $11$-SS's in the complement of each $12$-SS, and none of them can be completed to an $8$ Sidon-Ramsey partition of $[86]$ (it is not possible to add another $11$-SS's to any of this quatruples, and it is neither possible complete with a triple of $10$-SS's). If we don't use $12$-SS, we need to use at least six disjoint $11$-SS's. We found $4030$ sixtuples of disjoint $11$-SS's, none of whose complements contains a $10$-SS's nor an $11$-SS's. We conclude there is no $8$ Sidon-Ramsey partition of $[86]$.
\end{proof}

\begin{theorem}
	The Sidon-Ramsey number with $k=9$ satisfies $\SR(9)\leq 111$.
\end{theorem}

\begin{proof}
From Table \ref{table:sidonvalues} we know that the size of each part of a Sidon set in $[111]$ is at most $13$. There are twenty eight $13$-SS's in $[111]$ distributed as follows:
\begin{itemize}
 \item Two $13$-SS's in $[107]$ plus their eight translates.
 \item Six $13$-SS's in $[110]$ that contain $1$ and $110$, plus their six translates.
 \item Six $13$-SS's in $[111]$ that contain $1$ and $111$.
\end{itemize}
(There are no $13$-SS's in $[108]$ or $[109]$ that use both ends of the interval.)
In order to partition $[111]$ into $9$ Sidon sets we need at least three $13$-SS's. However, no triple amont the twenty eight $13$-SS's in $[111]$ above is disjoint.
\end{proof}

A Sidon-Ramsey partition is \textit{balanced} if the sizes of any pair of parts differ in at most $1$, and it is \textit{strongly-balanced} if the sizes are all the same. In \cite{She98} they try to find small intervals that contains many disjoint Sidon sets of the same size, not necessarily making a partition. When this constructions make a partition of some set $[n]$, it is an example of a strongly-balanced Sidon-Ramsey partition. In \cite{LLXX13} and \cite{XLL18} Sidon-Ramsey partitions of $[SR(k)-1]$ in $k$ parts are given, for $k\le 5$; they are all balanced partitions. 

The above discussion suggests that Sidon-Ramsey partitions of $[\SR(k)-1]$ are balanced, but as we shall see, this is not the case. So far we have found $5$ balanced Sidon-Ramsey partitions of $[65]$ with $7$ parts each:

\begin{align*}
 &\{1,3,15,22,30,33,46,50,55,56\},&
&\{2,9,17,21,26,27,47,49,60,63\},\\
&\{4,7,13,23,24,28,36,54,61\},&
&\{5,12,14,20,34,44,45,57,62\},\\
&\{6,10,16,29,38,41,43,58,59\},&
&\{8,25,31,32,35,40,51,53,65\},\\
&\{11,18,19,37,39,42,48,52,64\}.&
\end{align*}
\begin{align*}
 &\{1,3,15,22,30,33,46,50,55,56\},&
&\{2,6,14,24,27,29,38,57,58,64\},\\
&\{5,8,23,28,39,40,47,49,53\},&
&\{10,12,13,25,34,41,45,51,59\},\\
&\{7,16,20,21,36,42,44,54,61\},&
&\{9,11,17,32,43,48,52,62,65\},\\
&\{4,18,19,26,31,35,37,60,63\}.&
\end{align*}
\begin{align*}
 &\{2,4,16,23,31,34,47,51,56,57\},&
&\{7,10,15,19,33,43,44,50,63,65\},\\
&\{5,12,18,21,29,39,54,58,59\},&
&\{9,11,17,28,37,40,41,55,62\},\\
&\{8,13,14,22,24,42,45,49,64\},&
&\{3,20,26,27,30,35,46,48,60\},\\
&\{1,6,25,32,36,38,52,53,61\}.&
\end{align*}
\begin{align*}
 &\{2,4,16,23,31,34,47,51,56,57\},&
&\{3,9,10,29,38,40,43,53,61,65\},\\
&\{5,8,22,24,37,44,45,49,55\},&
&\{14,21,27,32,35,36,52,62,64\},\\
&\{7,11,12,26,39,42,48,50,60\},&
&\{1,13,19,20,28,30,33,54,58\},\\
&\{6,15,17,18,25,41,46,59,63\}.&
\end{align*}
\begin{align*}
 &\{3,5,17,24,32,35,48,52,57,58\},&
&\{2,6,14,16,19,37,38,44,53,64\},\\
&\{7,10,25,30,41,42,49,51,55\},&
&\{1,13,15,18,31,39,40,46,50\},\\
&\{4,21,22,26,28,36,47,56,59\},&
&\{8,11,23,27,29,34,54,62,63\},\\
&\{9,12,20,33,43,45,60,61,65\}.&
\end{align*}

These partitions have two parts of $10$ elements each, and five parts with $9$ elements each.
We conjecture that these and their reflected partitions (constructed by including the numbers $66-x$ instead of $x$ in each part) are all balanced Sidon-Ramsey partitions with those parameters. On the other hand, this is the first $k$ for which there are non-balanced Sidon-Ramsey partitions with $k$ parts in $[SR(k)-1]$. There is only one, namely

\begin{align*}
 &\{1,3,15,22,30,33,46,50,55,56\},&
&\{2,9,17,21,26,27,47,49,60,63\},\\
&\{4,7,12,16,31,41,42,48,62,64\},&
&\{5,11,14,28,32,39,44,52,54\},\\
&\{6,19,23,24,34,43,57,59,65\},&
&\{8,10,20,36,37,40,45,51,58\},\\
&\{13,18,25,29,35,38,53,61\}.&
\end{align*}
It has three parts with $10$ elements each, three parts with $9$ elements each and one part with $8$ elements.
Uniqueness was verified by determining all disjoint triples and quadruples of $10$-SS's in $[65]$ and analyzing one by one whether they could be extended to a Sidon-Ramsey partition with $7$ parts. This was enough, since there is no $5$-tuple of disjoint $10$-SS's in $[65]$.

The previous discussion naturally leads to the following open problems.

\begin{problem}
Is it true that for any positive integer $t$ there exists a Sidon-Ramsey partition in $k$ parts of $[\SR(k)-1]$ such that a pair of parts differs in size by $t$?
\end{problem}

\begin{problem}
Is there any positive integer $k$ such that there are more non-balanced Sidon-Ramsey partitions than balanced Sidon-Ramsey partitions in $k$ parts of $[\SR(k)-1]$?
\end{problem}

The Sidon-Ramsey partitions obtained so far show the following tendency when they are not strongly-balanced. The large Sidon sets of the partition do not include $1$ and $n$ simultaneously. Furthermore, if one of those large Sidon sets includes $1$ then the density of the small numbers of the part seems to be lower than the density of the large numbers of the part; for example, $\{1,3,15,22\}$ compared with $\{46,50,55,56\}$ in the first part of the first two balanced Sidon-Ramsey partitions of $[65]$. An analogous behavior can be observed for those large parts containing $n$ of a Sidon-Ramsey partition of $[n]$.

Intuitively, one can think that it is easier to find $k$-SS's in $[m+\ell]$ than in $[m]$ (assuming $\ell \ge 1$). This suggests that if we are told to bet on a given $d$-tuple of disjoint Sidon sets (say the large ones in the partition) so that they can be completed to a Sidon-Ramsey partition of $[n]$, then we should improve the chances of success if we manage to choose those $d$ sets so that they contain neither $1$ nor $n$. In that way, the remaining Sidon sets (say, the smaller ones) must be chosen within a larger interval (although the evidence given by the $5$ balanced partitions of $[65]$ shown above does not support this guess).

Based on the optimal Sidon-Ramsey partitions known so far, it seems like the chances of success are higher if the sets in the $d$-tuple mentioned above cover only a few numbers in the two ends of the interval $[n]$. For example, when we look for the numbers $\{1,2,3,4,5,61,62,63,64,65\}$ in the five balanced partitions of $[65]$ shown above (so that we take the five smallest and the five largest ones), the two large sets include $4$ of those extreme numbers in most of the cases, and only in one of them they include $5$ of these numbers. In comparison, a random pair of disjoing $10$-SS in $[65]$ contains $6$ or more of those numbers. The number $5$ for picking small and large numbers was chosen here since there are precisely $5$ more parts to be chosen (the Sidon sets with $9$ elements), but the evaluation is not very different if we choose the first and last six or seven numbers of $[65]$.
If true, this idea can be used to improve the lower bounds for $SR(k)$ by searching for large parts of the partition by favoring tuples that do not concentrate near $1$ or near $n$.

\section{Conclusions}\label{s:Conclusions}


While searching for Sidon-Ramsey partitions we realized how relevant it is to know all $k$-SS's in $[n]$ for the smallest values of $n$ for which they exist.

Denote by $g_k(n)$ the number of $k$-SS's in $[n]$ that contain $1$ and $n$. For $k \le 9$ the values of $g_k(n)$ for small $n$ strongly suggest that these are non-decreasing funcions.
It was striking to us that $g_{10}(n)$ is not non-decreasing, since $g_{10}(56)=2$ whereas $g_{10}(57)=g_{10}(58)=0$. Our intuition suggested us that if there are $10$-SS's in $[56]$ that use $1$ and $56$ then there should also be $10$-SS's in $[57]$ that use $1$ and $57$, since there is a little more space in $\{2,\dots,56\}$ to accomodate $8$ numbers to complete a $10$-SS with $1$ and $57$, in comparison with $\{2,\dots,55\}$ to be completed with $1$ and $56$.

The phenomenon of $g_k(n)$ decreasing to $0$ repeats with $g_{11}(n)$, $g_{12}(n)$ and $g_{13}(n)$. The function $g_{11}(n)$ equals $0$ if $n<73$, while $g_{11}(73)=4$ and $g_{11}(74)=0$ (for every $n>74$ the value of $g_{11}(n)$ is positive). Similarly, $g_{12}(n)=0$ if $n<86$ while $g_{12}(86)=2$, $g_{12}(87)=g_{12}(88)=g_{12}(89)=g_{12}(90)=0$, $g_{12}(91)=2$ and from there on $g_{12}(n)$ seems to be strictly increasing. Finally, $g_{13}(n)=0$ if $n<107$ while $g_{13}(107)=2$ and $g_{13}(108)=g_{13}(109)=0$ (curiously enough, $g_{13}(110)=g_{13}(111)=6$ so that $g_{13}$ is not even strictly increasing after the first two non-zero values).

\begin{problem}
Which are the numbers $k$ for which the function $g_k(n)$ just defined is non-decreasing?
\end{problem}

The number of known exact values of $SR(k)$ is very small to have much intuition of the nature of the Sidon-Ramsey partition attaining those numbers. Here we were able to improve the upper bounds of two more Sidon-ramsey numbers. We hope that soon new clever constructions of Sidon-Ramsey partitions will help improving the lower bounds as well (or suggesting that they are sharp).

\section*{Acknowledgments}

The second author was supported by PAPIIT-UNAM under project grant IN104021 and by
CONACYT ``Fondo Sectorial de Investigaci\'on para la Educaci\'on'' under grant A1-S-10839.

\bibliographystyle{amsplain}
\bibliography{refs.bib}

\end{document}